\documentclass[10pt,a4paper]{article}
\usepackage{amsfonts,amsmath,enumerate,amsthm,ifthen,xspace}
\usepackage[final]{graphicx}

\newcommand{\RootPath}{.}

\newcommand{\ExternalFiguresPath}{\RootPath/figures}

\newcommand{\eop}{\hspace*{\fill}~$\square$} 
\theoremstyle{plain}
\numberwithin{equation}{section}
\newtheorem{theorem}{Theorem}[section]
\newtheorem{conjecture}[theorem]{Conjecture}
\newtheorem{proposition}[theorem]{Proposition}
\newtheorem{lemma}[theorem]{Lemma}

\theoremstyle{definition}
\newtheorem{Remark}[theorem]{Remark}
\newenvironment{remark}{\begin{Remark}}{\eop\end{Remark}}

\newcommand{\mycaption}[1]{\centering{\vspace{\medskipamount}\refstepcounter{figure}\textbf{Figure~\thefigure.} {#1}}}

\newcommand{\natur}{\ensuremath{\mathbb{N}}}
\newcommand{\real}{\ensuremath{\mathbb{R}}}

\newcommand{\setcond}[2]{\left\{ #1 : #2 \right\}} 

\newcommand{\bigsetcond}[2]{\bigl\{ #1 \,  :\, #2 \bigr\}}
\newcommand{\Bigsetcond}[2]{\Bigl\{ #1 \,:\, #2 \Bigr\}}

\newenvironment{FigTab}[2]{
	\begin{figure}[htb]
	\setlength{\unitlength}{#2}
	\begin{center}
	\begin{tabular}{#1}
}{
    \end{tabular}
    \end{center}
    \end{figure}
}

\newcommand{\IncludeGraph}[2]{
	\includegraphics[#1]{\ExternalFiguresPath/{#2}}
}

\newcommand{\diam}{\mathop{\mathrm{diam}}\nolimits}

\newcommand{\intr}{\mathop{\mathrm{int}}\nolimits}

\newcommand{\overtwocond}[2]{\substack{{#1} \\ {#2}}} 
\newcommand{\dotvar}{\,\cdot\,} 

\newcommand{\MxL}{\left[} 
\newcommand{\MxR}{\right]} 

\newcommand{\ThmTitle}[2][]{\ifthenelse{\equal{#1}{}}{\emph{(#2)}}{\emph{(#2; #1)}}}
\newcommand{\notion}[2][]{\emph{#2}\xspace} 

\newcommand{\mfor}{\ \mbox{for} \ }
\newcommand{\mand}{\ \mbox{and} \ }

\newcommand{\sti}{\mathop{s_0}} 
\newcommand{\stic}{\mathop{s}} 

\newcommand{\eps}{\varepsilon}
\newcommand{\card}[1]{|#1|}

\newcommand{\cX}{X}

\newcommand{\mycite}[2]{\ifthenelse{\equal{#2}{}}{\cite{#1}}{\cite[#2]{#1}}\xspace}
\newcommand{\GroetschelHenk}[1][]{\mycite{MR1976602}{#1}}
\newcommand{\BosseGroetschelHenk}[1][]{\mycite{MR2166533}{#1}}
\newcommand{\Bernig}[1][]{\mycite{Bernig98}{#1}}
\newcommand{\Ziegler}[1][]{\mycite{MR1311028}{#1}}

\newcommand{\RAGbook}[1][]{\mycite{MR1659509}{#1}}
\newcommand{\BasuPollackRoy}[1][]{\mycite{MR2248869}{#1}}

\newcommand{\SchnBk}[1][]{\mycite{MR94d:52007}{#1}} 

\newcommand{\AverkovHenk}[1][]{\mycite{AveHenkRepSimplePolytopes}{#1}}

\begin{document}
\title{Representing Elementary Semi-Algebraic Sets\\ by a Few Polynomial Inequalities: \\A Constructive Approach}
\date{\small \today}
\author{\small Gennadiy Averkov\footnote{Work supported by the German Research Foundation within the Research Unit 468 ``Methods from Discrete Mathematics for the Synthesis and Control of Chemical Processes''.}}
\maketitle

\begin{abstract} 
	Let $P$ be an elementary closed semi-algebraic set in $\real^d$, i.e., there exist real polynomials $p_1,\ldots,p_s$ ($s \in \natur$) such that $P= \setcond{x \in \real^d}{p_1(x) \ge 0, \ldots, p_s(x) \ge 0}$; in this case $p_1,\ldots,p_s$ are said to represent $P$. Denote by $n$ the maximal number of the polynomials from $\{p_1,\ldots,p_s\}$ that vanish in a point of $P.$  If $P$ is non-empty and bounded, we show that it is possible to construct $n+1$ polynomials representing $P.$ Furthermore, the number $n+1$ can be reduced to $n$ in the case when the set of points of $P$ in which $n$ polynomials from $\{p_1,\ldots,p_s\}$ vanish is finite. Analogous statements are also obtained for elementary open semi-algebraic sets. 
\end{abstract} 

\newtheoremstyle{itdot}{}{}{\mdseries\rmfamily}{}{\itshape}{.}{ }{}
\theoremstyle{itdot}
\newtheorem*{msc*}{2000 Mathematics Subject Classification} 

\begin{msc*}
  Primary: 14P10, Secondary: 14Q99, 03C10, 90C26 
\end{msc*}

\newtheorem*{keywords*}{Key words and phrases}

\begin{keywords*}
Approximation, elementary symmetric function, {\L}ojasiewicz's Inequality, po\-lynomial optimization, semi-algebraic set, Theorem of Br\"{o}cker and Scheiderer
\end{keywords*}

\section{Introduction} 

In what follows $x:=(x_1,\ldots,x_d)$ is a variable vector in $\real^d$ ($d \in\natur$). As usual, $\real[x]:=\real[x_1,\ldots,x_d]$ denotes the ring of polynomials in variables $x_1,\ldots,x_d$ and coefficients in $\real.$ A subset $P$ of $\real^d$ which can be represented by
\begin{equation} \label{P:rep} 
	P= (p_1,\ldots,p_s)_{\ge 0}:=\setcond{x \in \real^d}{p_1(x) \ge 0,\ldots, p_s(x) \ge 0}
\end{equation}
for $p_1,\ldots,p_s \in \real[x]$ ($s \in \natur$) is said to be an \notion{elementary closed semi-algebraic set} in $\real^d.$ Clearly, the number $s$  from \eqref{P:rep} is not uniquely determined by $P.$ Let us denote by $\stic(d,P)$ the minimal $s$ such that \eqref{P:rep} is fulfilled for appropriate $p_1,\ldots,p_s \in \real[x].$ Analogously, a subset $P_0$ of $\real^d$ which can be represented by 
\begin{equation} \label{P0:rep} 
	P_0=(p_1,\ldots,p_s)_{> 0}:=\setcond{x \in \real^d}{p_1(x) > 0,\ldots, p_s(x) > 0}
\end{equation} 
for some $p_1,\ldots, p_s \in \real[x]$ ($s \in \natur$) is said to be an \notion{elementary open semi-algebraic set} in $\real^d.$ The quantity $\sti(d,P_0)$ associated to $P_0$ is introduced analogously to $\stic(d,P).$ The system of polynomials $p_1,\ldots,p_s$ from \eqref{P:rep}  (resp. \eqref{P0:rep}) is said to be a \notion[polynomial representation]{polynomial representation} of $P$ (resp. $P_0$).  From the well-known Theorem of Br\"ocker and Scheiderer (see \cite[Chapter~5]{MR1393194},  and \RAGbook[\S6.5, \S10.4] and the references therein) it follows that, for $P$ and $P_0$ as above, the following inequalities are fulfilled:
\begin{eqnarray}
	\stic(d,P) & \le & d(d+1)/2, \label{07.12.04,21:17} \\
	\sti(d,P_0) & \le & d.  \label{07.12.04,21:18}
\end{eqnarray} 
Both of these inequalities are sharp. It should be emphasized that all known proofs of \eqref{07.12.04,21:17} and \eqref{07.12.04,21:18} are highly non-constructive. The main aim of this paper is to provide constructive upper bounds for $\stic(d,P)$ and $\sti(d,P_0)$ for certain classes of $P$ and $P_0$; see also \cite{vomHofe}, \Bernig, \GroetschelHenk, \cite{Henk06PolRep}, \BosseGroetschelHenk, and \AverkovHenk for previous results on this topic. We also mention that constructive results on polynomial representations of special semi-algebraic sets are related to polynomial optimization; see \cite{MR1814045}, \cite{MR2011395}, \cite{MR2142861}, \cite{laurent-08}, and \cite{hel-nie-08}.

Let $p_1,\ldots,p_s \in \real[x]$ and let $P:=(p_1,\ldots,p_s)_{\ge 0}$ be non-empty. The assumptions of our main theorems are formulated in terms of the following functionals, which depend on $p_1,\ldots,p_s$. The functional
\begin{equation}
		I_x(p_1,\ldots,p_s) := \bigsetcond{i= 1,\ldots,s}{p_i(x) = 0}, \ x \in P, \label{07.11.22,17:53} \\
\end{equation}
determines the set of constraints defining $P$ which are ``active'' in $x.$ Furthermore, we define
	\begin{eqnarray}
		n(p_1,\ldots,p_s) &:=& \max \setcond{ \card{I_x(p_1,\ldots,p_s)} }{x \in P} , \label{07.11.22,17:54} \\
		\cX(p_1,\ldots,p_s) &:= & \bigsetcond{ x \in P}{ \card{I_x(p_1,\ldots,p_s)} =n(p_1,\ldots,p_s)}, \label{07.11.22,17:55}
	\end{eqnarray}
where $|\dotvar|$ stands for the cardinality.
 The geometric meaning of $n(p_1,\ldots,p_s)$ and $\cX(p_1,\ldots,p_s)$ can be illustrated by the following special situation. Let $P$ be a $d$-dimensional polytope with $s$ facets (see  \Ziegler 
 for information on polytopes). Then $P$ can be given by \eqref{P:rep} with all $p_i$ having degree one (the so-called \notion{H-representation}). In this case $n(p_1,\ldots,p_s)$ is the maximal number of facets of $P$ having a common vertex and $\cX(p_1,\ldots,p_s)$ is the set consisting of those vertices of $P$ which are contained in the maximal number of facets of $P.$ If the polytope $P$ is \notion{simple} (that is, each vertex of $P$ lies in precisely $d$ facets), then $n(p_1,\ldots,p_s)=d$ and $\cX(p_1,\ldots,p_s)$ is the set of all vertices of $P.$ 
 
  Now we are ready to formulate our main results.

\begin{theorem} \label{main:n+1}
	Let $p_1,\ldots, p_s \in \real[x]$, $P := (p_1,\ldots,p_s)_{\ge 0}$, and $P_0 := (p_1,\ldots,p_s)_{> 0}.$ 
	Assume that $P$ is non-empty and bounded, and $n:=n(p_1,\ldots,p_s) < s.$ Then the following inequalities are fulfilled:
	\begin{equation*}
		\begin{array}{ccc}
		\stic(d,P) \le n+1, & &
		\sti(d,P_0) \le n+1
		\end{array}
	\end{equation*}
	Furthermore, there exists an algorithm that gets $p_1,\ldots,p_s$ and returns $n+1$ polynomials $q_0,\ldots,q_{n} \in \real[x]$ satisfying $P = (q_0,\ldots,q_{n})_{\ge 0}$ and  $P_0 = (q_0,\ldots,q_{n})_{> 0}.$	\eop
\end{theorem}

In the case when $\cX(p_1,\ldots,p_s)$ is finite Theorem~\ref{main:n+1} can be improved. 

\begin{theorem} \label{main:n}
	Let $p_1,\ldots, p_s \in \real[x]$, $P := (p_1,\ldots,p_s)_{\ge 0}$, and $P_0 := (p_1,\ldots,p_s)_{> 0}.$ 	Assume that $P$ is non-empty and bounded,  $\cX:=\cX(p_1,\ldots,p_s)$ is finite, and $n:=n(p_1,\ldots,p_s) < s$. Then the following inequalities are fulfilled: 
	\begin{equation*}
		\begin{array}{ccc}
		\stic(d,P) \le n, & &
		\sti(d,P_0) \le n
		\end{array}
	\end{equation*}
	Furthermore, there exists an algorithm that gets $p_1,\ldots,p_s$ and $\cX$ and returns $n$ polynomials $q_1,\ldots,q_{n}$ satisfying $P = (q_1,\ldots,q_{n})_{\ge 0}$ and  $P_0 = (q_1,\ldots,q_{n})_{> 0}.$	\eop
\end{theorem} 

 Below we discuss existing results and problems related to Theorems~\ref{main:n+1} and \ref{main:n}. Let $P$ be a convex polygon in $\real^2$ with $s$ edges,  which is given by  \eqref{P:rep} with all $p_i$ having degree one. Bernig \Bernig  showed that  setting $q_2:=p_1 \cdot \ldots \cdot p_s$ one can construct a strictly concave polynomial $q_1(x)$ vanishing on all vertices of $P$ which satisfies $P= (q_1,q_2)_{\ge 0}$ ; see Fig.~\ref{simp:polyt:2d:fig}. As it will be seen from the proof of Theorem~\ref{main:n}, for the case $d=2$ and $P$ as in Theorem~\ref{main:n} we also set $q_2 := p_1 \cdot \ldots \cdot p_s$ and choose $q_1$ in such a way that it vanishes on each point of $X$ and the set $(q_1)_{\ge 0}$ approximates $P$ sufficiently well; see Fig.~\ref{Fig:semi-alg:2d}. However, since $P$ from Theorem~\ref{main:n} is in general not convex, the construction of $q_1$ requires a different idea. The statement of Theorem~\ref{main:n} concerned with $P_0$ and restricted to the cases $n=2$ and $n=d$, $s=d+1$ (with slightly different assumptions on $P_0$) was obtained by Bernig \Bernig[Theorems~4.1.1 and 4.3.5].

	\begin{FigTab}{c}{0.6mm}
	\begin{picture}(200,50)
	\put(60,2){\IncludeGraph{width=50\unitlength}{BernigSetP1.eps}}
	\put(0,2){\IncludeGraph{width=50\unitlength}{BernigSetP0.eps}}
	\put(150,2){\IncludeGraph{width=50\unitlength}{BernigP.eps}}
	\put(125,25){\IncludeGraph{width=12\unitlength}{RightArrow.eps}}
	\put(20,25){\scriptsize $(q_1)_{\ge 0}$}
	\put(80,25){\scriptsize $(q_2)_{\ge 0}$}
	\put(170,25){\scriptsize $P$}
	\end{picture}
	\\
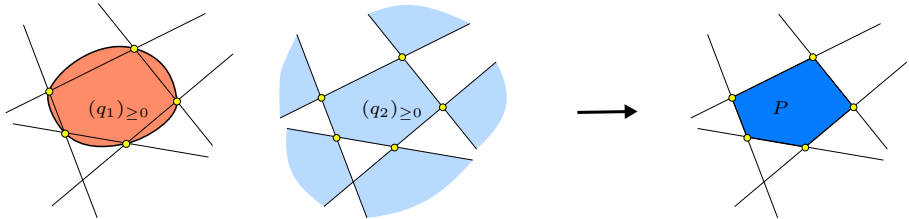
	\parbox[t]{0.98\textwidth}{\mycaption{Illustration to the result of Bernig on convex polygons\label{simp:polyt:2d:fig}}} 
	\end{FigTab}

	\begin{FigTab}{c}{0.7mm}
	\begin{picture}(200,50)
	\put(60,2){\IncludeGraph{width=50\unitlength}{SetP1.eps}}
	\put(0,2){\IncludeGraph{width=50\unitlength}{SetP0.eps}}
	\put(150,2){\IncludeGraph{width=50\unitlength}{SetP.eps}}
	\put(125,25){\IncludeGraph{width=12\unitlength}{RightArrow.eps}}
	\put(20,25){\scriptsize $(q_1)_{\ge 0}$}
	\put(80,25){\scriptsize $(q_2)_{\ge 0}$}
	\put(170,25){\scriptsize $P$}
	\end{picture} \\
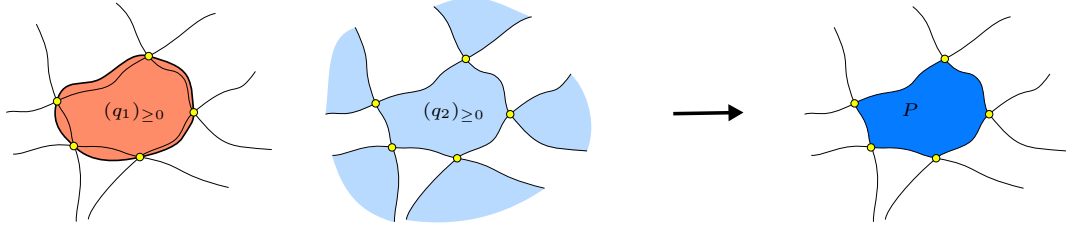
	\parbox[t]{0.95\textwidth}{\mycaption{Illustration to Theorem~\ref{main:n} for the case $d=2$, $n=2$\label{Fig:semi-alg:2d}}} 
	\end{FigTab}

The study of $\stic(d,P)$ for the case when $P$ is a polyhedron of an arbitrary dimension was initiated by Gr\"otschel and Henk \GroetschelHenk. In \GroetschelHenk[Corollary~2.2(i)] it was noticed that $\stic(d,P) \ge d$ for every $d$-dimensional polytope $P.$ On the other hand, Bosse, Gr\"otschel, and Henk \BosseGroetschelHenk gave an upper bound for $\stic(d,P)$ which is linear in $d$ for the case of an arbitrary $d$-dimensional polyhedron $P.$ In particular, they showed that $\stic(d,P) \le 2d-1$ if $P$ is  $d$-dimensional polytope.  In \BosseGroetschelHenk the following conjecture was announced.

\begin{conjecture} \ThmTitle{Bosse \& Gr\"otschel \& Henk 2005} \label{bgh:conj}
	  For every $d$-dimensional polytope $P$ in $\real^d$  the equality $\stic(d,P)=d$ holds. \eop
\end{conjecture} 
This conjecture has recently been confirmed for all simple $d$-dimensional polytopes; see \AverkovHenk. 

\begin{theorem} \ThmTitle{Averkov \& Henk 2007+} \label{AveHenkThm}
	Let $P$ be a $d$-dimensional simple polytope  Then $\stic(d,P) =d.$ Furthermore, there exists an algorithm that gets polynomials  $p_1,\ldots,p_s$ ($s \in \natur$) of degree one satisfying $P=(p_1,\ldots,p_s)_{\ge 0}$ and returns $d$ polynomials $q_1,\ldots,q_d$ satisfying
	$P = (q_1,\ldots,q_d)_{\ge 0}.$ 
	\eop
\end{theorem}

Elementary closed semi-algebraic sets $P:=(p_1,\ldots,p_s)_{\ge 0}$ with $n(p_1,\ldots,p_s)=d$ can be viewed as natural extensions of simple polytopes in the framework of real algebraic geometry. Thus, we can see that Theorem~\ref{AveHenkThm} is a consequence of Theorem~\ref{main:n}. Fig.~\ref{cube:p:rep} illustrates Theorem~\ref{AveHenkThm} for the case when $P$ is a three-dimensional cube.   This figure can also serve as an illustration of Theorem~\ref{main:n} with the only difference that in Theorem~\ref{main:n} the set $(p_1)_{\ge 0}$ does not have to be convex anymore.

\begin{FigTab}{c}{0.5mm}
\begin{picture}(132,160)
\put(35,55){\IncludeGraph{width=55\unitlength}{Cube.eps}}
\put(-15,30){\IncludeGraph{width=55\unitlength}{CubeP0.eps}}
\put(90,30){\IncludeGraph{width=55\unitlength}{CubeP1.eps}}
\put(35,120){\IncludeGraph{width=55\unitlength}{CubeP2.eps}}
\put(-15,90){\IncludeGraph{width=55\unitlength}{CubeP02.eps}}
\put(90,90){\IncludeGraph{width=55\unitlength}{CubeP12.eps}}
\put(35,0){\IncludeGraph{width=55\unitlength}{CubeP01.eps}}
\put(5,10){\IncludeGraph{width=110\unitlength}{PDiagram.eps}}
\put(45,55){\scriptsize $P$}
\put(-5,30){\scriptsize $(q_1)_{\ge0}$}
\put(100,30){\scriptsize $(q_2)_{\ge 0}$}
\put(45,120){\scriptsize $(q_3)_{\ge 0}$}
\put(-5,90){\scriptsize $(q_1,q_3)_{\ge 0}$}
\put(100,90){\scriptsize $(q_2,q_3)_{\ge 0}$}
\put(45,0){\scriptsize $(q_1,q_2)_{\ge 0}$}
\end{picture}
\\
\parbox[t]{0.98\textwidth}{\mycaption{Illustration to Theorem~\ref{AveHenkThm} (and Theorem~\ref{main:n}) for the case when $P$ is a three-dimensional cube.\label{cube:p:rep}}}
\end{FigTab}

While proving our main theorems we derive the following approximation results which can be of independent interest. \newcommand{\disth}{\mathop{\mathrm{\delta}}}
The \notion{Hausdorff distance} $\disth$ is a metric defined on the space of non-empty compact subsets of $\real^d$ by the equality
\begin{equation*}
	\disth(A,B) := \max \Bigl\{ \max_{a \in A} \min_{b \in B} \|a-b\|, \max_{b \in B} \min_{a \in A} \|a-b\| \Bigr\},
\end{equation*}
see \SchnBk[p.~48]. 

\begin{theorem} \label{approx:thm1} Let $p_1,\ldots,p_s \in \real[x],$ $P:=(p_1,\ldots,p_s)_{\ge 0}$, and $P_0:=(p_1,\ldots,p_s)_{>0}.$ Assume that $P$ is non-empty and bounded. Then there exists an algorithm that gets $p_1,\ldots,p_s$ and $\eps>0$ and returns a polynomial $q \in \real[x]$ such that  $P_0 \subseteq (q)_{>0}$, $P \subseteq (q)_{\ge 0}$, and the Hausdorff distance from $P$ to $(q)_{\ge 0}$ is at most $\eps.$ \eop
\end{theorem} 

\begin{theorem} \label{approx:thm2} Let $p_1,\ldots,p_s \in \real[x],$ $P:=(p_1,\ldots,p_s)_{\ge 0}$, and $P_0:=(p_1,\ldots,p_s)_{>0}.$ Assume that $P$ is non-empty and bounded, $X:=X(p_1,\ldots,p_s)$ is finite, and $n:=n(p_1,\ldots,p_s)<s.$ Then there exists an algorithm that gets $p_1,\ldots,p_s$, $X$, and $\eps>0$ and returns a polynomial $q \in \real[x]$ such that  $P_0 \subseteq (q)_{>0}$, $P \subseteq (q)_{\ge 0}$, the Hausdorff distance from $P$ to $(q)_{\ge 0}$ is at most $\eps$, and $q(x)=0$ for every $x \in X.$ \eop
\end{theorem} 

We note that some further results on approximation by sublevel sets of polynomials can be found  in \cite{MR0154184}, \cite{MR0353146}, and \GroetschelHenk[Lemma~2.6].

 The paper has the following structure. Section~\ref{prelim:sect} contains preliminaries from real algebraic geometry.  In Section~\ref{approx:sect} we obtain  approximation results (including Theorems~\ref{approx:thm1} and \ref{approx:thm2}). Finally, in Section~\ref{main:proofs:sect} the proofs of Theorems~\ref{main:n+1} and \ref{main:n}  are presented. In the beginning of the proofs of Theorems~\ref{main:n+1} and \ref{main:n} one can find the formulas defining the polynomials $q_i$ (see \eqref{07.12.04,10:58} and \eqref{08.04.02,14:45}) as well as sketches of the main arguments. 

%

\section{Preliminaries from real algebraic geometry} \label{prelim:sect}

The origin and the Euclidean norm in $\real^d$ are denoted by $o$ and $\|\dotvar\|,$ respectively. We endow $\real^d$ with its Euclidean topology. By $B^d(c,\rho)$ we denote the closed Euclidean ball in $\real^d$ with center at $c \in \real^d$ and radius $\rho >0.$ The interior (of a set) is abbreviated by $\intr.$ We also define $\natur_0:=\natur \cup \{0\},$ where $\natur$ is the set of all natural numbers.  

A set $A \subseteq \real^d$ given by 
$$
	A:=\bigcup_{i=1}^k \setcond{x \in \real^d}{f_{i,1}(x) > 0, \ldots, f_{i,s_i}(x)>0, \  g_i(x)=0},
$$
where $i \in \{1,\ldots,k\}, \ j \in \{1,\ldots,s_i\}$ and $f_{i,j}, g_i \in \real[x]$,  is called \notion[semi-algebraic set]{semi-algebraic}.

\newcommand{\Free}{\mathop{\mathrm{Free}}\nolimits}

An expression $\Phi$ is called a \notion{first-order formula  over the language of ordered fields with coefficients in $\real$} if $\Phi$ is a formula built with a finite number of conjunctions, disjunctions, negations, and universal or existential quantifier on variables, starting from formulas of the form $f(x_1,\ldots,x_d)=0$ or $g(x_1,\ldots,x_d)>0$ with $f, g \in \real[x]$; see \RAGbook[Definition~2.2.3].  The \notion{free variables} of $\Phi$ are those variables, which are not quantified. A formula with no free variables is called a \emph{sentence}. Each sentence is is either true or false.  The following proposition is well-known; see also \RAGbook[Proposition~2.2.4] and \BasuPollackRoy[Corollary~2.75]. 
\begin{proposition} \label{semialg:over:formula}
	Let $\Phi$ be a first-order formula over the language of ordered fields with coefficients in $\real$ and free variables $ y_1,\ldots,y_m.$ Then the set 
	\begin{equation*}
		\setcond{ (y_1,\ldots,y_m) \in \real^m}{ \Phi(y_1,\ldots,y_m) },
	\end{equation*}
	consisting of all $(y_1,\ldots,y_m) \in \real^d$ for which $\Phi$ is true, is semi-algebraic. \eop
\end{proposition} 

A real valued function $f(x)$ defined on a semi-algebraic set $A$ is said to be a \notion{semi-algebraic function} if its graph is a semi-algebraic set in $\real^{d+1}.$  The following theorem presents \notion{{\L}ojasiewicz's Inequality}; see \cite{MR0107168} and \RAGbook[Corollary~2.6.7].

\begin{theorem}  \ThmTitle{{\L}ojasiewicz 1959} \label{Loj}
	Let $A$ be non-empty, bounded, and closed semi-algebraic set in $\real^d.$ Let $f$ and $g$ be continuous, semi-algebraic functions defined on $A$ and such that   $\setcond{x \in A}{f(x) =0} \subseteq \setcond{x \in A}{g(x) = 0}.$
	Then there exist $M \in \natur$ and $\lambda \ge 0$ such that 
	\begin{equation*}
		|g(x)|^M \le \lambda \, |f(x)| 
	\end{equation*}
	for every $x \in A.$  \eop
\end{theorem}

\newcommand{\cP}{\mathcal{P}}

Considering algorithmic questions we use the following standard settings; see \cite[Chapter~\S8.1]{MR1393194}.  It is assumed that a polynomial in $\real[x]$ is given by its coefficients and that a finite list of real coefficients occupies finite memory space. Furthermore, arithmetic and  comparison operations over reals are assumed to be atomic, i.e., computable in one step. The following well-known result is relevant for the constructive part of our  theorems; see \BasuPollackRoy[Algorithm~12.30]. 

\begin{theorem} \label{dec:problem} \ThmTitle{Tarski 1951, Seidenberg 1954}
Let $\Phi$ be a sentence over the language of ordered fields with coefficients in $\real$. Then there exists an algorithm that gets $\Phi$ and decides whether $\Phi$ is true or false. \eop
\end{theorem}

\section{Approximation results} \label{approx:sect}

The following proposition (see \SchnBk[p.~57]) presents a characterization of the convergence with respect to the Hausdorff distance. 
\begin{proposition} \label{Hausd:convergence} 
	A sequence $(A_n)_{n=1}^{+\infty}$ of compact convex sets in $\real^d$ converges to a compact set $A$ in the Hausdorff distance if and only if the following conditions are fulfilled:
	\begin{enumerate} 
		\item Every point of $A$ is a limit of a sequence $(a_k)_{k=1}^{+\infty}$ satisfying $a_k \in A_k$ for every $k \in \natur.$ 
		\item If $(k_j)_{j=1}^{+\infty}$ is a strictly increasing sequence of natural numbers and $(a_{k_j})_{j=1}^{+\infty}$ is a convergent sequence satisfying $a_{k_j} \in A_{k_j}$ ($j \in \natur$), then $a_{k_j}$ converges to a point of $A$, as $j \rightarrow +\infty.$
		\item The set $\bigcup_{k=1}^{+\infty} A_k$ is bounded. 
	\end{enumerate} 
	\eop
\end{proposition}

Let $p_1,\ldots,p_s \in \real[x].$ The following theorem states that for the case when $P:=(p_1,\ldots,p_s)_{\ge 0}$ is non-empty and bounded, appropriately relaxing the inequalities $p_i(x) \ge 0$, which define $P$, we get a bounded semi-algebraic set that approximates $P$ arbitrarily well. 
Let us define
\begin{equation} \label{S:eps:def}
   P(M, \eps) := \setcond{x \in \real^d}{ (1+\|x\|^2)^M p_i(x) \ge -\eps \mfor 1 \le i \le s}
 \end{equation}
with $M \in \natur_0$ and $\eps>0.$ 
\begin{theorem} \label{semi:approx}
 Let $p_1,\ldots,p_s \in \real[x]$, $P:=(p_1,\ldots,p_s)_{\ge 0}$, and $P_0:=(p_1,\ldots,p_s)_{>0}.$ Assume that $P$ is non-empty and bounded.  Then there exists an algorithm that gets $p_1,\ldots,p_s$ and returns values $M \in \natur_0$ and $\eps_0 >0$ such that the following conditions are fulfilled: 
\begin{enumerate}[1.] 
 \item \label{boundedness:part}   $P(M,\eps)$ is bounded for $\eps =\eps_0.$
\item \label{convergence:part} $P(M,\eps), \ \eps \in (0,\eps_0],$ converges to $P$ in the Hausdorff distance, as $\eps \rightarrow 0.$ 
\end{enumerate} 
\eop
\end{theorem}
\begin{proof} First we show the existence of $M$ and $\eps_0$ from the assertion, and after this we show that these two quantities are constructible. Let us derive the existence of $M$ and $\eps_0$ satisfying Condition~\ref{boundedness:part}. Since $P$ is bounded, after replacing $P$ by an appropriate homothetical copy, we may assume that $P \subseteq \intr B^d(o,1).$ By Proposition~\ref{semialg:over:formula}, the function $$f(x) := - \min_{1 \le i \le s} p_i(x)$$  is semi-algebraic. We also have $f(x)>0$ for all $x \in \real^d$ with $\|x\| \ge 1.$ Furthermore, the set $P(M,\eps)$ can be expressed with the help of $f(x)$ by  
\begin{equation} \label{07.12.11,16:37}
	P(M,\eps) = \setcond{x \in \real^d}{ (1+\|x\|^2)^M f(x) \le \eps}.
\end{equation} 
For $t \ge 1$ the function $$a(t) := \min \setcond{ f(x) }{1 \le \|x\| \le t}$$ is positive and non-increasing. Using Proposition~\ref{Hausd:convergence} it can be shown that $a(t)$ is continuous. Moreover, in view of Proposition~\ref{semialg:over:formula}, we see that $a(t)$ is semi-algebraic.  In  the case $\inf \setcond{a(t)}{t \ge 1} > 0$ Condition~\ref{boundedness:part} is fulfilled for $M=0$ and $\eps_0 = \frac{1}{2} \inf \setcond{a(t)}{t \ge 1}.$ In the opposite case we have $a(t) \rightarrow 0$, as $t \rightarrow +\infty.$ Then
 \begin{equation*}
   b(t) := \begin{cases} a(1/t), & 0 < t \le 1, \\ 0, & t=0 \end{cases}
 \end{equation*}
 is a continuous semi-algebraic function on $[0,1]$ with $b(t)=0$ if and only if $t=0.$ Thus, applying Theorem~\ref{Loj} to the functions $b(t)$ and $t^2$ defined on $[0,1]$, we see that there exist $M \in \natur_0$ and $\gamma > 0$ such that $t^{2M} \le \gamma \, b(t)$ for every $t \in [0,1].$ Consequently $t^{2M} a(t) \ge \frac{1}{\gamma}$ for every $t \ge 1.$ The latter implies that $(1+ \|x\|^2)^M f(x) \ge \frac{1}{\gamma},$ and Condition~\ref{boundedness:part} is fulfilled for $M$ as above and $\eps_0 = \frac{1}{2 \gamma}.$   Now we show that Condition~\ref{boundedness:part} implies Condition~\ref{convergence:part}. Assume that Condition~\ref{boundedness:part} is fulfilled. Then the set $P(M,\eps)$ is bounded for all $\eps \in (0,\eps_0].$ Hence $\disth(P,P(M,\eps))$ is well defined for all $\eps \in (0,\eps_0].$ Consider an arbitrary sequence $(t_j)_{j=1}^{+\infty}$ with $t_j \in (0,\eps_0]$ and $t_j \rightarrow 0,$ as $j \rightarrow + \infty,$  using Proposition~\ref{Hausd:convergence} we can see that $\disth(P,P(t_j)) \rightarrow 0,$ as $j \rightarrow +\infty.$ Consequently, Condition~\ref{convergence:part} is fulfilled.

Finally we show that $\eps_0$ and $M$ are constructible. For determination of $M$ one can use the following ``brute force'' procedure.
\begin{description} 
	\item[Procedure:] Determination of $M.$
	\item[Input:] $p_1,\ldots,p_s  \in \real[x].$ 
	\item[Output:] A number $M \in \natur_0$ such that for some $\eps_0>0$ the set $P(M,\eps_0)$ is bounded.
\end{description}  
\begin{enumerate}[1:] 
	\item Set $M := 0.$ 
	\item For $i \in \{1,\ldots,s \}$ introduce the first-order formula $$\Phi_i:=" (1+x_1^2+ \cdots +x_d^2)^M \, p_i(x_1,\ldots,x_d) \ge -\eps_0 "$$
	with free variables $x_1,\ldots,x_d, \eps_0.$
	\item Test the existence of $\eps_0>0$ for which $P(M,\eps_0)$ is bounded. More precisely, determine whether the sentence 
	$$
		\Psi:="(\exists \eps_0 ) (\exists \tau)  \ \ (\eps_0>0) \wedge (\forall x_1) \ldots (\forall x_d) \left( \Phi_1 \wedge \ldots \wedge \Phi_s \rightarrow (x_1^2+\cdots + x_d^2 \le \tau^2 ) \right) "
	$$
	is true or false (cf. Theorem~\ref{dec:problem}).
	\item If $\Psi$ is true, return $M$ and stop. Otherwise set $M:=M + 1$ and go to Step~2. 
\end{enumerate} 
In view of the conclusions made in the proof, the above procedure terminates after a finite number of iterations. For determination of $\eps_0$ we can use a similar procedure. We start with $\eps_0:=1$ and assign $\eps_0:=\eps_0/2$ at each new iteration, terminating the cycle as long as $P(M,\eps_0)$ is bounded.  
\end{proof} 

\begin{remark} \label{bounded:extension:remark}
We wish to show Theorem~\ref{semi:approx} cannot be improved by setting $M:=0,$ since $P(0,\eps)$ may be unbounded for all $\eps>0.$ Let us consider the following example. Let  $M=0,$  $d=2,$ $s=1,$ and $$p_1(x) = - (x_1-x_2)^2 - (x_1^2+x_2^2-1) \, (1+x_1^2 - x_2^2)^2.$$ Then the set $P=(p_1)_{\ge 0}$ is bounded. In fact, if $\|x\|>1,$ then the term  $x_1^2+x_2^2-1$, appearing in the definition of $p_1$, is positive. But the remaining terms $x_1-x_2$ and $1+x_1^2-x_2^2$ cannot vanish simultaneously. Hence, $p_1(x) <0$ for every $x$ with $\|x\|>1,$ which shows that $P \subseteq B^2(o,1).$ Furthermore, since $p_1(o)<0,$ we see that $P$ has non-empty interior (which shows that our example is non-degenerate enough). Let us show that $P(M,\eps)= \setcond{x \in \real^2}{q_1(x) \ge -\eps}$ is unbounded for every $\eps >0.$ For $x(t) := (t, \sqrt{1+t^2} )$ with $t \ge 0$ one has $\|x(t) \| = \sqrt{1+ 2t^2} \rightarrow +\infty$ and $p_1(x(t)) = - \left(t - \sqrt{1+t^2} \right)^2\rightarrow 0^{-},$ as $t \rightarrow +\infty$; see also Fig.~\ref{bd:ubd:fig}. This implies unboundedness of $P(M,\eps).$ 
\end{remark}

Throughout the rest of the paper we shall use the following polynomials associated to $p_1,\ldots,p_s \in \real[x].$ For $M \in \natur_0$, $\lambda>0$, and $k \in \natur$ we define 
	\begin{equation} \label{g:def}
		g_{M,\lambda,k}(x):= \frac{1}{s} \sum_{i=1}^s \Bigl(1 - \frac{1}{\lambda} (1+ \|x\|^2)^M \, p_i(x) \Bigr)^{2k}
	\end{equation}
If $X:=X(p_1,\ldots,p_s)$ is finite, we define 
	\begin{equation*}
	h_{\mu}(x) := \prod_{v \in \cX} \left(\frac{\|x-v\|}{\mu}\right)^2,
	\end{equation*}
where $\mu>0.$ 

\begin{lemma} \label{sublevel:approx:part} Let $p_1,\ldots,p_s \in \real[x]$, $P:=(p_1,\ldots,p_s)_{\ge 0}$, and $P_0:=(p_1,\ldots,p_s)_{>0}.$ Assume that $P$ is non-empty and bounded.  Then for every $\eps>0$, $M \in \natur_0$, $\lambda>0,$ and $k \in \natur$ satisfying 
	\begin{eqnarray} \label{07.11.23,17:44}
		\lambda & \ge & \max_{1 \le i \le s} \max_{x \in P} \, (1+ \|x\|^2)^M p_i(x),  \\
		s & \le &  \left(1+  \frac{\eps}{\lambda} \right)^{2k}\label{m:eps:cond} 
	\end{eqnarray} 
	the polynomial $g(x):=g_{M,\lambda,k}(x)$ fulfills the relations
\begin{eqnarray} 
	 P_0  & \subseteq  & \setcond{x \in \real^d}{ g(x) < 1}  \subseteq P(M,\eps), \label{08.04.08,16:28} \\
	P & \subseteq & \setcond{x \in \real^d}{ g(x) \le 1} \subseteq P(M,\eps). \label{07.11.12,16:41}
\end{eqnarray}
Furthermore, there exists an algorithm that gets $p_1,\ldots,p_s$, $\eps>0$, and $M \in \natur_0$ and constructs $g=g_{M,\lambda,k} \in \real[x]$ satisfying \eqref{08.04.08,16:28} and \eqref{07.11.12,16:41}.
\end{lemma}
\begin{proof}
Inclusions $P_0 \subseteq \setcond{x \in \real^d}{ g(x) < 1}$ and $P \subseteq  \setcond{x \in \real^d}{ g(x) \le 1}$ follow from \eqref{07.11.23,17:44}. It remains to show the inclusion $\setcond{x \in \real^d}{ g(x) \le 1} \subseteq P(M,\eps).$ Assume that $g(x) \le 1. $ Then 
$$\max_{1 \le i \le s}  \Bigl(1 - \frac{1}{\lambda} (1+ \|x\|^2)^M \, p_i(x) \Bigr)^{2k} \le s \stackrel{\eqref{m:eps:cond}}{\le} \left(1+ \frac{\eps}{\lambda}\right)^{2k}.$$ 
Consequently 
$$\max_{1 \le i \le s} \Bigl(1 - \frac{1}{\lambda} (1+ \|x\|^2)^M \, p_i(x) \Bigr) \le 1 +  \frac{\eps}{\lambda},$$
or equivalently, $(1+\|x\|^2)^M f(x) \le \eps$. Hence $x \in P(M,\eps).$ 

Now let us discuss the constructibility of $g(x).$ It suffices to show the constructibility of $\lambda$ satisfying \eqref{07.11.23,17:44}. For determination of $\lambda$ we iterate starting with $\lambda:=1$, set $\lambda:=\lambda+1$ at each new step, and use \eqref{07.11.23,17:44}, reformulated as a first-order formula, as a condition for terminating  the cycle. 
\end{proof}

	\begin{FigTab}{c}{0.9mm}
	\begin{picture}(80,70)
	\put(0,2){\IncludeGraph{width=70\unitlength}{bounded-unbounded.eps}}
	\put(70,10){$x_1$}
	\put(10,70){$x_2$}
	\put(36,38){$P$}
	\end{picture} \\
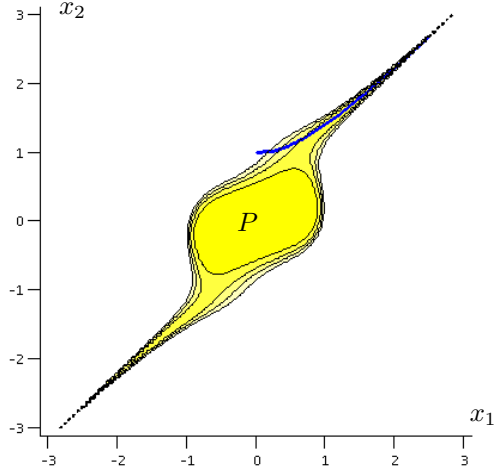
	\parbox[t]{0.95\textwidth}{\mycaption{Illustration to Remark~\ref{bounded:extension:remark}: the level sets given by equations $p_1(x)=0,$ $p_1(x)=-0.3,$ $p_1(x)=-0.5$, $p_1(x)=-0.7$ and a part of the curve with parametrization $x(t)$\label{bd:ubd:fig}}} 
	\end{FigTab}

One can see that Theorem~\ref{approx:thm1} from the introduction is a direct consequence of Theorem~\ref{semi:approx} and Lemma~\ref{sublevel:approx:part}.



\begin{theorem} \label{07.12.10,11:53}
	Let $p_1,\ldots,p_s \in \real^d$, $P:=(p_1,\ldots,p_s)_{\ge 0}$, and $P_0:=(p_1,\ldots,p_s)_{>0}$. Assume that $P$ is non-empty and bounded,  $\cX:=\cX(p_1,\ldots,p_s)$ is finite, and $n:=n(p_1,\ldots,p_s) < s.$ Then there exists an algorithm that gets $p_1,\ldots,p_s,$ $X,$ $M \in \natur_0$, and $\eps>0$ and returns $q \in \real[x]$ fulfilling the relations
	\begin{equation*}
	\begin{array}{rcccl}
		P_0 & \subseteq & (q)_{>0} & \subseteq & P(M,2  \eps) , \\
		P & \subseteq & (q)_{\ge 0} & \subseteq & P(M,2  \eps), \\
			& &	X & \subseteq & \setcond{x \in \real^d}{q(x)=0}.  
	\end{array}
	\end{equation*}  

	Furthermore, $q$ can be defined by 
	$$
		q(x):=\sigma_{s-n+1}(p_1(x),\ldots,p_s(x)) - g_{M,\lambda,k}(x)^l h_{\mu}(x)^m,
	$$
	where $k, l, m  \in \natur$, $\lambda>0$, and 
	$\mu>0.$ 
	\eop
\end{theorem}
\begin{proof}
	Analogously to the proof of Theorem~\ref{semi:approx}, we first show the existence of $q$ from the assertion and then we derive the constructive part of the theorem. We fix $\lambda$ and $k$ satisfying \eqref{07.11.23,17:44} and \eqref{m:eps:cond} and set $g(x):=g_{M,\lambda,k}(x).$ Let us derive the inclusions 
	$P_0 \subseteq (q)_{>0}$ and $P \subseteq (q)_{\ge 0}$. First we show that 
	\begin{equation} \label{g:bound}
		\max_{x \in P} g(x) < 1. 
	\end{equation}
	Let $I_x:=I_x(p_1,\ldots,p_s).$ Since $n<s,$ for every $x \in P$ the set $I_x$ is properly contained in $\{1,\ldots,s\}$. Consequently, for every $x \in P$ we get
	$$	
		g(x) = \frac{1}{s} \left( \card{I_x} + \sum_{i \in \{1,\ldots,s\} \setminus I_x} \left(1- \frac{1}{\lambda} (1+\|x\|^2)^M p_i(x)\right)^{2k} \right) < 1. 
	$$
	Thus, \eqref{g:bound} is fulfilled. Therefore we can fix $\alpha$ with 
	\begin{equation}
	\label{alpha<1}
		\max_{x \in P} g(x) \le \alpha < 1.
	\end{equation} 
	In view of \eqref{alpha<1} and the finiteness of $\cX,$  we can fix $\rho>0$ such that 
	\begin{equation}  \label{a:07.11.12,16:43}
		\bigcup_{v \in \cX} B^d(v,\rho) \subseteq \setcond{x \in \real^d}{g(x) \le 1}. 
	\end{equation} 
	and 
	\begin{equation} \label{disj:balls}
	B^d(v,\rho) \cap B^d(w,\rho) = \emptyset
	\end{equation} 
	 for all $v, w \in \cX$ with $v \ne w.$ 

	Let us consider an arbitrary $x \in P.$ We show that, for an appropriate choice of $l \in \natur$ and $m \in \natur$ we have $q(x) \ge 0,$ and the latter inequality is strict for $x \in P_0.$ 

	\emph{Case~A:} $x \in P \cap \left( \bigcup_{v \in \cX} B^d(v,\rho) \right).$ Let us fix $w \in\cX$ such that $\|x-w\| \le \rho.$  Since $x \in P$, we have $\sigma_{s-n+1}(p_1(x),\ldots,p_s(x))\ge 0.$ Furthermore, due to the choice of $\rho,$ equality is attained if and only if $x=w.$  Let $\mu > 0$ be an arbitrary scalar satisfying
	\begin{equation} \label{mu:bound} 
		\mu \ge \diam(P) := \max \setcond{ \|x'-x''\| }{x', x'' \in P}.
	\end{equation} 
	Applying Theorem~\ref{Loj} to the functions $\sigma_{s-n+1}(p_1(x),\ldots,p_s(x))$ and $\left(\frac{\|x-w\|}{\mu}\right)^2$ restricted to $B^d(w,\rho) \cap P$, we have 
	\begin{equation*} 
		\left(\frac{\|x-w\|}{\mu}\right)^{2 m(w)} \le \tau(w) \cdot \sigma_{s-n+1}(p_1(x),\ldots,p_s(x)) 
	\end{equation*} 
	for appropriate parameters $\tau(w) > 0$ and $m(w) \in \natur$ independent of $x.$ In view of the choice of $\mu$ we deduce
\begin{equation} \label{a:07.11.12,17:03}
  \left(\frac{\|x-w\|}{\mu}\right)^{2 m} \le \tau \cdot \sigma_{s-n+1}(p_1(x),\ldots,p_s(x)),
\end{equation} 
where $\tau := \max_{v \in \cX} \tau(v)$ and $m:=\max_{v \in \cX} m(v).$ 
	We have 
	\begin{align}
		g(x)^l h_\mu(x)^m & \stackrel{\eqref{alpha<1}}{\le} \alpha^l \, h_\mu(x)^m = \alpha^l \, \left(\frac{\|x-w\|}{\mu}\right)^{2m} \, \prod_{v \in \cX \setminus \{w\}} \left(\frac{\|x-v\|}{\mu}\right)^{2m} \stackrel{\eqref{mu:bound}}{\le}  \alpha^l \left( \frac{\|x-w\|}{\mu} \right)^{2m}  \nonumber \\
					& \stackrel{\eqref{a:07.11.12,17:03}}{\le}  \tau \, \alpha^l  \, \sigma_{s-n+1}(p_1(x),\ldots,p_s(x)). \label{a:07.12.04,11:00}
	\end{align}
	In view of  \eqref{alpha<1}, for all sufficiently large $l \in \natur$ the inequality
	\begin{equation} \label{a:07.11.15,18:56}
		\tau \, \alpha^l < 1,
	\end{equation}
	 is fulfilled. Assuming that \eqref{a:07.11.15,18:56} holds, and taking into account \eqref{a:07.12.04,11:00}, we have $q(x) \ge 0.$  

	Now assume that  $x$ lies in $P_0 \cap \left( \bigcup_{v \in \cX} B^d(v,\rho) \right).$ Then, if $l$ satisfies \eqref{a:07.11.15,18:56}, we get  $q(x) > 0.$ 

	\emph{Case~B:} $ x \in P \setminus \bigcup_{v \in \cX} B^d(v,\rho).$ Then $\|x-v\| \ge \rho$ for every $v \in X.$ From the definition of elementary symmetric functions and the assumptions it easily follows that
	\begin{equation*}
		\min \Bigsetcond{ \sigma_{s-n+1}(p_1(x'),\ldots,p_s(x'))}{x' \in P \setminus \bigcup_{v \in \cX} \intr B^d(v,\rho)} > 0.
	\end{equation*} 
	Let us choose $\gamma$ with 
	\begin{equation}  \label{gamma:bound}
		0  < \gamma \le  \min \Bigsetcond{ \sigma_{s-n+1}(p_1(x'),\ldots,p_s(x'))}{x' \in P \setminus \bigcup_{v \in \cX} \intr B^d(v,\rho)} .
	\end{equation} 
	Thus, we get the bounds $$g(x)^l h_\mu(x)^m \stackrel{\eqref{alpha<1}}{\le} \alpha^l h_\mu(x)^m \stackrel{\eqref{mu:bound}}{\le}  \alpha^l$$ and $\gamma \le \sigma_{s-n+1}(p_1(x),\ldots,p_s(x)).$ 		
 In view of \eqref{alpha<1}, for all sufficiently large $l \in \natur$ the inequality
	\begin{equation} \label{a:07.11.15,18:57}
		\alpha^l < \gamma
	\end{equation} 
	is fulfilled. Assuming that \eqref{a:07.11.15,18:57} is fulfilled, we obtain $q(x) > 0.$ 
	
	Now we show the inclusion $(q)_{\ge 0} \subseteq P(M,2 \eps)$. Consider an arbitrary $x \in \real^d \setminus P(M,2\eps).$ Then $$\min_{1 \le i \le s} (1+\|x\|^2)^M \, p_i(x)  \le -2 \, \eps,$$ which is equivalent to 
	\begin{equation} \label{08.04.04,17:34}
		\max_{1 \le i \le s} \Biggl( 1- \frac{1}{\lambda} (1+\|x\|^2)^M \, p_i(x) \Biggr) \ge 1+   \frac{2 \eps}{\lambda}.
	\end{equation}
	The latter implies that 
	$$
		\sum_{i=1}^s \Bigl(1- \frac{1}{\lambda} (1+\|x\|^2)^M \, p_i(x) \Bigl)^{2k} \ge \left(1+  \frac{2 \eps}{\lambda}\right)^{2k},
	$$
	and therefore
	\begin{equation} \label{a:g:eps:bound}
		g(x) \ge \frac{1}{s} \left(1+ \frac{2 \eps}{\lambda} \right)^{2k} \stackrel{\eqref{m:eps:cond}}{\ge} \left( \frac{\lambda+2 \eps}{\lambda + \eps} \right)^{2k} > 1. 
	\end{equation}
	We have 
	$$
	\begin{array}{rcl}
		\bigl| \sigma_{s-n+1}(p_1(x),\ldots,p_s(x)) \bigr| & \le & \sigma_{s-n+1}(|p_1(x)|,\ldots,|p_s(x)|)\\
			& \le &  \sigma_{s-n+1}(\underbrace{1,\ldots,1}_s) \max\limits_{1 \le j \le s} |p_j(x)|^{s-n+1}\\
	& = & \binom{s}{n-1} \, \max\limits_{1 \le j \le s}   |p_j(x)|^{s-n+1} \\ 
	& \le & \binom{s}{n-1} \, \lambda^{s-n+1} \max\limits_{1 \le j \le s} \Bigl| \frac{1}{\lambda} \, (1+\|x\|^2)^M \, p_j(x) \Bigr|^{s-n+1}  \\	
		&  \le &    \binom{s}{n-1} \, \lambda^{s-n+1} \left( \max\limits_{1 \le j \le s} \Bigl| 1- \frac{1}{\lambda} (1+\|x\|^2)^M p_j(x) \Bigr| + 1 \right)^{s-n+1} \\ 
		&  \stackrel{\eqref{08.04.04,17:34}}{\le} &    \binom{s}{n-1} \, \lambda^{s-n+1} \left( \max\limits_{1 \le j \le s} \Bigl| 1- \frac{1}{\lambda} (1+\|x\|^2)^M p_j(x) \Bigr|^{2k} + 1 \right)^{s-n+1} \\ 
	& \stackrel{\eqref{a:g:eps:bound}}{\le} & \binom{s}{n-1} \, \lambda^{s-n+1} \bigl( s \, g(x) + 1 \bigr)^{s-n+1} \\ 
	& \stackrel{\eqref{a:g:eps:bound}}{\le} & \binom{s}{n-1} \, \lambda^{s-n+1} \bigl( s \, g(x) + g(x) \bigr)^{s-n+1} \\
& = & \binom{s}{n-1} \lambda^{s-n+1} (s+1)^{s-n+1} g(x)^{s-n+1}.
	\end{array}
	$$
	The above estimate for $|\sigma_{s-n+1}(p_1(x),\ldots,p_s(x))|$ together with the estimate  $$h_\mu(x)^m = \prod_{v \in \cX} \left(\frac{\|x-v\|}{\mu}\right)^{2m} \ge \left(\frac{\rho}{\mu}\right)^{2 m \, \card{\cX}} > 0$$ and \eqref{a:g:eps:bound} implies that 
	$|\sigma_{s-n+1}(p_1(x),\ldots,p_s(x)) | \le \frac{1}{2} g(x)^l h_\mu(x)^m$ if $l$ fulfills the inequality 
	\begin{equation} \label{a:07.11.15,18:58}
		2 \, \binom{s}{n-1} \, \lambda^{s-n+1} (s+1)^{s-n+1} \le \left( \frac{\lambda + 2  \eps}{\lambda + \eps} \right)^{2k \, (l-s+n-1)} \, \left(\frac{\rho}{\mu}\right)^{2 m \card{\cX}}.
	\end{equation}
	Since $\frac{\lambda + 2 \eps}{\lambda +  \eps} > 1$,  \eqref{a:07.11.15,18:58} is fulfilled if $l \in \natur$ is large enough. Thus, we obtain that the inequality $q(x) < 0$ holds for all sufficiently large $l.$ 

	Now we show the constructive part of the assertion. We present a sketch of a possible procedure that determines $q.$  It suffices to evaluate the parameters $k, l, m, \lambda,$ and  $\mu$ involved in the definition of $q$. Constructibility of $\lambda$ and $k$ follows from Lemma~\ref{sublevel:approx:part}. Let us apply Theorem~\ref{dec:problem} in the same way as in the previous proofs. Determine the following parameters in the given sequence. We can determine $m$ satisfying \eqref{a:07.11.12,17:03} for an appropriate $\tau >0$ and all $x \in P \cap \left(\bigcup_{v \in \cX} B^d(v,\rho) \right)$ using the same idea as in the procedure for determination of $M$ in the proof of Theorem~\ref{semi:approx}.  A parameter $\mu$ satisfying \eqref{mu:bound} is constructible in view of Theorem~\ref{dec:problem} (by means of iteration procedure which we also used in the previous proofs). An appropriate $l$ can be easily found from inequalities \eqref{a:07.11.15,18:56}, \eqref{a:07.11.15,18:57}, and \eqref{a:07.11.15,18:58}. Thus, for evaluation of $l$ we should first find the parameters $\tau, \alpha,$ and $\rho$ appearing in \eqref{a:07.11.15,18:56}, \eqref{a:07.11.15,18:57}, and \eqref{a:07.11.15,18:58}. The parameters $\alpha$, $\tau$,  and $\gamma$ are determined by means of \eqref{alpha<1},  \eqref{a:07.11.12,17:03}, and \eqref{gamma:bound}.
\end{proof}

One can see that Theorem~\ref{approx:thm2} from the introduction is a straightforward consequence of Theorem~\ref{semi:approx} and Theorem~\ref{07.12.10,11:53}.
	

\begin{remark} 
The parameters $k, l, m, M, \lambda, \mu$ involved in the statements of this section were computed with the help of the Theorem~\ref{dec:problem}. In contrast to this, in general  it is not possible to compute $\cX$ exactly, since evaluation of $\cX$ would involve solving a polynomial system of equations. This explains why in the statement of Theorem~\ref{07.12.10,11:53} the set $\cX$ is taken as a part of the input. 
\end{remark}

\begin{remark} 
	The parameters $\lambda$ and $\mu$ from Lemma~\ref{sublevel:approx:part} and Theorem~\ref{07.12.10,11:53}, respectively, are upper bounds for certain polynomial programs. In fact, by \eqref{07.11.23,17:44} the parameter $\lambda>0$ is a common upper bound for the optimal solutions of $s$ non-linear programs $p_i(x) \rightarrow \max, \ i \in \{1,\ldots,s\},$ with constraints $p_j(x) \ge 0$, $1 \le j \le s.$  From the proof of Theorem~\ref{07.12.10,11:53} we see that $\mu$ can be any number satisfying $\mu \ge \diam(P).$ Hence $\mu^2$ is an upper bound for the optimal solution of the polynomial program $\|x'-x''\|^2 \rightarrow \max, \ x', x'' \in \real^d,$ with $2d$ unknowns (which are coordinates of $x'$ and $x''$) and the $2s$ constraints $p_i(x') \ge 0$ and $p_i(x'') \ge 0$,  $1 \le i \le s.$ The same observations apply also to the parameters $\alpha$ and $\gamma$ from the proof of Theorem~\ref{07.12.10,11:53}, which are used for determination of $l$.        In this respect we notice that upper bounds of polynomial programs can be determined using convex relaxation methods; see \cite{MR1814045}, \cite{MR2011395}, and \cite{MR2142861}.
\end{remark}

\section{Proofs of the main theorems} \label{main:proofs:sect}

Given $s \in \natur$, $k \in \{1,\ldots,s\},$ and $y:=(y_1,\ldots,y_s) \in \real^s$ the \notion[elementary symmetric function]{$k$-th elementary symmetric function} in variables $y_1,\ldots,y_s$  is defined by
\begin{equation} \label{elsym:def} 
 \sigma_k(y) := \sum_{\overtwocond{I \subseteq \{1,\ldots,s\}}{\card{I} = k}} \ \prod_{i \in I} y_i.
\end{equation}
We also put $\sigma_0(y):=1.$ 
\begin{proposition}  \ThmTitle{Bernig 1998} \label{BernigLemma} 
	Let $y:=(y_1,\ldots,y_s) \in \real^s$ with $s \in \natur.$ Then the following statements hold: 
	\begin{enumerate}[I.]
	\item \label{nonstrict:part} $y_1 \ge 0, \ldots, y_s \ge 0$ if and only if $\sigma_1(y) \ge 0, \ldots, \sigma_s(y) \ge 0$.
	\item \label{strict:part} $y_1 > 0, \ldots, y_s > 0$ if and only if $\sigma_1(y) > 0, \ldots, \sigma_s(y) >0$.
	\end{enumerate}
\end{proposition} 
\begin{proof} 
The necessities of both of the parts are trivial. Let us prove the sufficiencies. We introduce the polynomial $f(t) = (t+y_1) \cdot \ldots \cdot (t+y_s),$ whose roots are the the values $-y_1,\ldots -y_s.$ By \notion{Vieta's formulas}, we have $f(t) = \sigma_s(y) \, t^0 + \sigma_{s-1}(y) \, t^1 + \cdots + \sigma_0(y) \, t^s.$ Thus, if $\sigma_i(y) \ge 0$ for every $i \in \{1,\ldots,s\}$, then all coefficients of $f(t)$ are non-negative, while the coefficient at $t^s$ is equal to one. It follows that $f(t)$ cannot have strictly positive roots. Hence $y_i \ge 0$ for all $i \in \{1,\ldots,s\},$  which shows the sufficiency of Part~\ref{nonstrict:part}. Now assume that the strict inequality $\sigma_i(y) > 0$ holds for every $i \in \{1,\ldots,s\}.$ Then $f(0)=\sigma_s(y) >0,$ i.e., zero is not a root of $f(t),$ and, using the sufficiency of Part~\ref{nonstrict:part}, we arrive a the strict inequalities $y_1 > 0,\ldots, y_s >0.$ This shows the sufficiency in Part~\ref{strict:part}.
\end{proof}
Proposition~\ref{BernigLemma} was noticed by Bernig \Bernig[p.~38], who derived it from \notion{Descartes' Rule of Signs}. Our elementary proof (slightly) extends the arguments given in \AverkovHenk. 

\begin{lemma} \label{determ:n}
	Let $p_1,\ldots,p_s \in \real[x]$ and $P:=(p_1,\ldots,p_s)_{\ge 0}.$ Assume that $P$ is non-empty and bounded. Then there exists an algorithm which gets $p_1,\ldots,p_s$ and returns $n(p_1,\ldots,p_s).$ 
\end{lemma}
\begin{proof}
	Since $P$ is bounded, we have $n(p_1,\ldots,p_s) \le 1.$ We suggest the following procedure for evaluation of $n(p_1,\ldots,p_s).$
	\begin{description} 
		\item[Procedure:] Evaluation of $n(p_1,\ldots,p_s)$
		\item[Input:] $p_1,\ldots,p_s \in \real[x].$
		\item[Output:] $n(p_1,\ldots,p_s)$
	\end{description}  
	\begin{enumerate}[1:]
		\item For $i =1,\ldots,s$ introduce the formula $$\Phi_i:= " p_i(x_1,\ldots,x_d) \ge 0"$$
		with free variables $x_1,\ldots,x_d.$ 
		\item Set $n:=1.$ 
		\item \label{iter:begin} Introduce the formula
		$$
		 \Phi:="\prod_{\overtwocond{J \subseteq \{1,\ldots,s\}}{\card{J} = n}} \, \, \sum_{j \in J} p_j(x_1,\ldots,x_d)^2 = 0"
		$$
		with free variables $x_1,\ldots,x_d.$ 
		\item Verify whether the sentence 
		$$
			\Psi:= "(\exists x_1) \ldots (\exists x_d) \ \Phi \wedge \Phi_1 \wedge \ldots \wedge \Phi_s" 
		$$
		is true or not.
		\item If $\Psi$ is true and $n<s,$ set $n:=n+1$ and go to Step~\ref{iter:begin}.
		\item If $\Psi$ is true and $n=s,$ return $n$ and stop. 
		\item If $\Psi$ is false, set $n:=n-1$, return $n$, and stop 
	\end{enumerate} 
	It is not hard to see that the above procedure terminates in a finite number of steps and returns $n(p_1,\ldots,p_s).$ 
\end{proof} 

\begin{proof}[Proof of Theorem~\ref{main:n+1}]
		As in the previous proofs, we first show the existence of $q_0,\ldots, q_n$ from the assertion and then discuss the algorithmic part. We define $q_i, \ 0 \le i \le n,$ by the formula
	\begin{equation} \label{07.12.04,10:58}
		q_i(x)  :=   \begin{cases} 1-g_{M,\lambda,k}(x) & \mfor i=0,  \\ \sigma_{s-n+i}(p_1(x),\ldots,p_s(x)) & \mfor 1 \le i \le n, \end{cases} 
	\end{equation}
	where $k \in \natur$, $M \in \natur_0$, and $\lambda>0$ will be fixed later. (We recall that $g_{M,\lambda,k}(x)$ is defined by \eqref{g:def}.)
	Let us first present a brief sketch of our arguments. It turns out that the polynomials $q_1,\ldots,q_{n},$ which are defined with the help of elementary symmetric functions, represent $P$ locally, that is,  $P$ and $(q_1,\ldots,q_{n})_{\ge 0}$ coincide in a neighborhood of $P.$ In order to pass to the global representation, the additional polynomial $q_0$ is  chosen in such a way that the sublevel set $(q_0)_{\ge 0}$ approximates $P$ sufficiently well. 

		Given $\eps>0$ let us consider the set $P(M,\eps)$ defined by \eqref{S:eps:def}. By Theorem~\ref{semi:approx} there exist $M \in \natur_0$  and $\eps_0>0$ such that $P(M,\eps_0)$ is bounded.  Since $n<s$ it follows that $\sigma_i(p_1(x),\ldots,p_s(x))>0$ for all $x \in P$ and $1 \le i \le s-n.$ Thus, the above strict inequalities hold also for $x$ in a small neighborhood of $P.$ Consequently, by Theorem~\ref{semi:approx}, we can fix an $\eps \in (0,\eps_0]$  such that $\sigma_i(p_1(x),\ldots,p_s(x))>0$ for all $x \in P(M,\eps)$ and $1 \le i \le s-n.$ We define the sets
	\begin{align*} 
	\begin{array}{lcr}
	Q := \setcond{ x\in \real^d}{q_i(x) \ge 0 \mfor 0 \le i \le n} &\mand& Q_0 := \setcond{ x \in \real^d}{q_i(x) > 0 \mfor 0 \le i \le n}.
	\end{array}
	\end{align*}
	Let us consider an arbitrary $x \in P.$ Obviously, $q_i(x) \ge 0$ for $1 \le i \le n,$ where all inequalities are strict if $x \in P_0.$ 
	Assume that $\lambda$ and $k$ satisfy \eqref{07.11.23,17:44} and \eqref{m:eps:cond}. Then, by Lemma~\ref{sublevel:approx:part}, $q_0(x) \ge 0,$ where the inequality is strict if $x \in P_0.$ Hence $P \subseteq Q$  and $P_0 \subseteq Q_0.$ Let us show the reverse inclusions. Let $x \in Q_0.$ Then, by the definition of $q_0,\ldots,q_n,$ we have $\sigma_i(p_1(x),\ldots,p_s(x)) > 0$ for $s-n +1 \le i \le s$ and $g_{M,\lambda,k}(x) < 1.$ But, by the choice of $\eps$ and $g_{M,\lambda,k}(x)$, we also have $\sigma_i(p_1(x),\ldots,p_s(x)) > 0$ for $1 \le i \le s-n.$ Thus, $\sigma_i(p_1(x),\ldots,p_s(x)) > 0$ for $1 \le i \le s,$ and, in view of Proposition~\ref{BernigLemma}(\ref{strict:part}), we have $p_i(x) > 0$ for $1 \le i \le s.$ This shows the inclusion $Q_0 \subseteq P_0.$ The inclusion $Q \subseteq P$ can shown analogously (by means of  Proposition~\ref{BernigLemma}(\ref{nonstrict:part})). 

	Finally we discuss the constructive part of the statement. By Lemma~\ref{determ:n}, $n$ is computable. Consequently, the polynomials $q_1,\ldots,q_n$ are also computable, since they are arithmetic expressions in $p_1,\ldots,p_s.$ The computability of $q_0$ follows from directly from Theorem~\ref{semi:approx}. 
\end{proof}

\begin{proof}[Proof of Theorem~\ref{main:n}] 
	The polynomials $q_1,\ldots, q_i$ will be defined by 
	\begin{equation} \label{08.04.02,14:45}
		q_i(x)  := \begin{cases} \sigma_{s-n+1}(p_1(x),\ldots,p_s(x)) -   g_{M,\lambda,k}(x)^l  h_\mu(x)^m & \mfor i=1,  \\ \sigma_{s-n+i}(p_1(x),\ldots,p_s(x)) & \mfor 2 \le i \le n, \end{cases} 
	\end{equation}
	where $k, l, m \in \natur, \ M \in \natur_0, \ \lambda >0, \ \mu >0$  will be fixed below.

We give a rough description of the arguments. We start with the same remark as in the proof of Theorem~\ref{main:n+1}. Namely, polynomials $\sigma_j(p_1(x),\ldots,p_s(x))$ with $s-n +1 \le j \le s$ represent $P$ locally. We shall disturb the polynomial $\sigma_{s-n+1}(p_1(x),\ldots,p_s(x))$ by subtracting an appropriate non-negative polynomial $g_{M,\lambda,k}(x)^l h_\mu(x)^m$ which is small on $P$, has high order zeros at the points of $X,$ and is large for all points $x$ sufficiently far away from $P.$  See also Fig.~\ref{Fig:semi-alg:2d} for an illustration of Theorem~\ref{main:n} in the case $d=2.$

	We first show the existence of $q_1,\ldots,q_n$ from the assertion. Given $\eps>0,$ let us consider the set $P(M,\eps)$ defined by \eqref{S:eps:def}. By Theorem~\ref{semi:approx} there exist $M \in \natur_0$  and $\eps_0>0$ such that $P(M,\eps_0)$ is bounded.  Since $n<s$ it follows that $\sigma_i(p_1(x),\ldots,p_s(x))>0$ for all $x \in P$ and $1 \le i \le s-n.$ Thus, the above strict inequalities hold also for $x$ in a small neighborhood of $P.$ Consequently, by Theorem~\ref{semi:approx}, we can fix $\eps \in (0,\eps_0/2]$  such that $\sigma_i(p_1(x),\ldots,p_s(x))>0$ for all $x \in P(2  \eps)$ and $1 \le i \le s-n.$ Let us borrow the notations from the statements of Theorems~\ref{semi:approx} and \ref{07.12.10,11:53}.

	We set $q_1:=q$ with $q \in \real[x]$ as in Theorem~\ref{07.12.10,11:53}. Define the semi-algebraic sets
	\begin{align*} 
	\begin{array}{lcr}
		Q = (q_1,\ldots,q_n)_{\ge 0} &\mand& Q_0 :=(q_1,\ldots,q_n)_{>0}.
	\end{array}
	\end{align*}
	
	Let us consider an arbitrary $x \in P.$ Obviously, $q_i(x) \ge 0$ for $2 \le i \le n,$ where all inequalities are strict if $x \in P_0.$ Furthermore, by Theorem~\ref{07.12.10,11:53} we also have $q_1(x) \ge 0$ and this inequality is strict if $x \in P_0.$ Thus, we get the inclusions $P \subseteq Q$ and $P_0 \subseteq Q_0.$ 

	It remains to verify the inclusions $Q \subseteq P$ and $Q_0 \subseteq P_0.$ Let us consider an arbitrary $x \in \real^d \setminus P_0,$ that is, for some $i \in \{1, \ldots,s\}$ one has $p_i(x) \le 0.$ If $x \in P(2 \eps) \setminus P_0,$ then, by the choice of $\eps,$ $\sigma_i(p_1(x),\ldots,p_s(x)) > 0$ for all $1 \le i \le s-n.$ But, on the other hand, by Proposition~\ref{BernigLemma}(\ref{strict:part}), $\sigma_j(p_1(x),\ldots,p_s(x)) \le 0$ for some $1 \le j \le s.$ Hence we necessarily have $j > s-n$, and we get that $q_{j+n-s}(x) \le 0.$ Consequently $x \in \real^d \setminus Q_0.$  Now assume $x \in P(2 \eps) \setminus P.$ Then, by Proposition~\ref{BernigLemma}(\ref{nonstrict:part}), $\sigma_j(p_1(x),\ldots,p_s(x)) < 0$ for some $1 \le j \le s.$ But, in the same way as we showed above, we deduce that $j > s-n.$ Hence $q_{j+n-s}(x) < 0,$ which means that $x \in \real^d \setminus Q.$ If $x \in \real^d \setminus P(2 \eps),$ then, by Theorem~\ref{07.12.10,11:53}, one has $q_1(x) <0,$ and by this $x  \in \real^d \setminus Q.$ 

	As for the algorithmic part of the assertion, we notice that $n=n(p_1,\ldots,p_s)$ can be easily computed from $X.$ The computability of $q_1$ follows from Theorem~\ref{07.12.10,11:53}. 
\end{proof}

\begin{remark} 
We mention that the ``combinatorial component'' of our proofs (dealing with elementary symmetric functions) resembles in part the proof of Theorem~\ref{AveHenkThm}. However, the crucial parts of the proofs of Theorems~\ref{main:n+1} and Theorem~\ref{main:n} concerning the approximation of $P$ are based on different ideas. The polynomials $q_1,\ldots,q_d$ from Theorem~\ref{AveHenkThm} can be computed in a rather straightforward way; see \AverkovHenk[Section~4]. In contrast to this, the constructive parts of the proofs of Theorems~\ref{main:n+1} and \ref{main:n} use decidability of the first order logic over reals and, by this, lead to algorithms of extremely high complexity. Even though Theorem~\ref{Loj} and Proposition~\ref{BernigLemma} were also used in \Bernig, our proofs cannot be viewed as extensions of the proofs from \Bernig.
\end{remark}

\section*{Acknowledgements} I am indebted to Prof. Martin Henk for his support during the preparation of the manuscript. The examples in  Remark~\ref{bounded:extension:remark} arose from a discussion with Prof. Claus Scheiderer.


\bibliographystyle{amsalpha}
\providecommand{\bysame}{\leavevmode\hbox to3em{\hrulefill}\thinspace}
\providecommand{\MR}{\relax\ifhmode\unskip\space\fi MR }
\providecommand{\MRhref}[2]{%
  \href{http://www.ams.org/mathscinet-getitem?mr=#1}{#2}
}
\providecommand{\href}[2]{#2}

 \begin{tabular}{l}
        \textsc{Gennadiy Averkov, 
	Universit\"atsplatz 2, Institut f\"ur Algebra und Geometrie,
        } \\
        \textsc{Fakult\"at f\"ur  Mathematik, Otto-von-Guericke-Universit\"at Magdeburg,}
	\\
	\textsc{D-39106 Magdeburg}
	\\
        \emph{e-mail}: \texttt{gennadiy.averkov@googlemail.com} \\
	\emph{web}: \texttt{fma2.math.uni-magdeburg.de/$\sim$averkov}
    \end{tabular} 
 \end{document}